\documentclass{article}

\usepackage{amsmath, amssymb}
\usepackage{graphicx, wrapfig} 
\usepackage{amsthm}
\usepackage{url}

\theoremstyle{plain}
\newtheorem{thm}{Theorem}
\newtheorem{lem}{Lemma}

\theoremstyle{definition}

\newtheorem{prop}{Proposition}
\newtheorem{cor}{Corollary}

\title{Region crossing change on origami and link}
\author{Tokio Oshikiri\thanks{College of Science and Engineering, Kanazawa University, Kakuma-machi, Kanazawa-shi, Ishikawa 920-1192, Japan.} 
\and Ayaka Shimizu\thanks{Osaka Central Advanced Mathematical Institute (OCAMI), Osaka Metropolitan University, 3-3-138, Sugimoto, Osaka 558-8585, Japan. shimizu1984@gmail.com}
\and Junya Tamura\thanks{Taiyo Yuden Co., LTD. Tamamura Plant, 1796-1, Kawai, Tamamura-machi, Gunma 370-1196, Japan. } }
\date{\today}
\begin{document}
\maketitle

\begin{abstract}
Region Select is a game originally defined on a knot projection. 
In this paper, Region Select on an origami crease pattern is introduced and investigated. 
As an application, a new unlinking number associated with region crossing change is defined and discussed.
\end{abstract}

\section{Introduction}

Let $D$ be a knot diagram on $S^2$. 
Let $R$ be a region of $D$, namely a connected portion of $S^2$ divided by $D$. 
A {\it region crossing change} at a region $R$ is an operation on $D$ which changes all the crossings on the boundary of $R$. 
In this paper, we denote a region crossing change by the abbreviation RCC. 

{\it Region Select} is a game on a knot projection which has a lamp with the status of ON or OFF at each crossing with the following rule. 
When one selects a region, then all the lamps on the boundary of the region are switched. 
The goal of the game is to light up all the lamps for given knot projection with an initial status of lamps (see Figure \ref{fig-rs-k}). 
Using results of the study of RCC, it has been shown that the Region Select is solvable for any knot projection with lamps.

\begin{figure}[h]
\centering
\includegraphics[width=100mm]{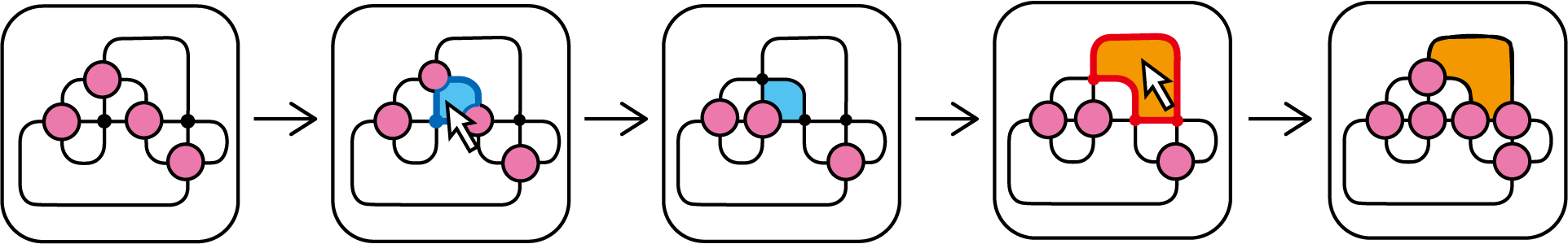}
\caption{Region Select on a knot projection.}
\label{fig-rs-k}
\end{figure}

In this paper, the Region Select on origami is defined and discussed.
By folding a square sheet of paper $S$ and open it, we have a crease pattern, and by giving a lamp at each intersection point of creases, we obtain a game board of the Region Select. 
Region Select on origami is not always solvable. 
For example, we cannot clear the game depicted in Figure \ref{fig-ors-ex}. 
In Section \ref{section-ors}, some conditions to switch a single lamp on crease patterns with vertices of degree four will be discussed.

\begin{figure}[h]
\centering
\includegraphics[width=16mm]{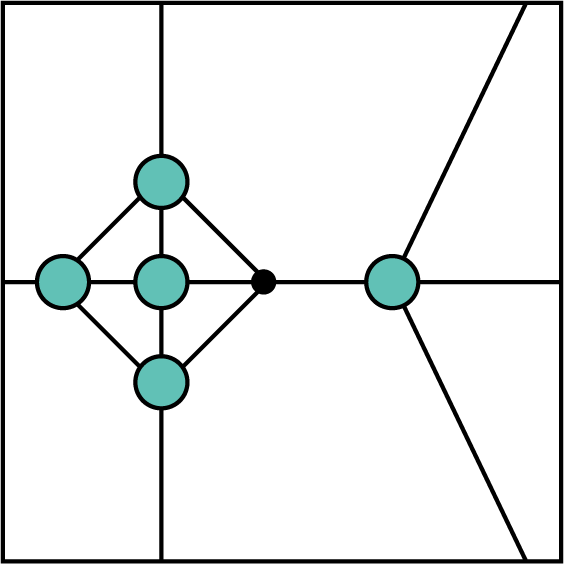}
\caption{Region Select on a crease pattern of origami.}
\label{fig-ors-ex}
\end{figure}

In \cite{rcc-uo}, it was shown that a RCC is an unknotting operation on a knot diagram. 
Then, the {\it region unknotting number}, $u_R(D)$, of a knot diagram $D$ was defined to be the minimal number of RCCs to obtain a diagram of the trivial knot from $D$. 
In \cite{aida}, it was shown by Aida that any knot has a diagram which is transformed into a diagram of the trivial knot by a single ``$n$-gon move''. 
Since an $n$-gon move is a kind of region crossing change, we have the following: 

\vspace{2mm}
\begin{thm}[\cite{aida}]
Every nontrivial knot has a diagram whose region unknotting number is one. 
\end{thm}

\vspace{2mm}

\noindent For links, a RCC is not an unlinking operation, that is, we cannot use RCC as a tool to transform any link into a trivial link (\cite{cheng-gao}). 
In this paper, using the concept of origami Region Select, we will define the {\it circled region unlinking number} for link diagrams in Section \ref{section-li} and prove the following theorem: 

\vspace{2mm}

\begin{thm}
Every nontrivial link has a diagram whose circled region unlinking number is one. 
\label{thm-cru-one}
\end{thm}

\vspace{2mm}

\noindent The rest of the paper is organized as follows: 
In Section \ref{section-rcc}, RCC on knot diagrams and some properties are discussed. 
In Section \ref{section-ors}, Region Select game on origami crease patterns is proposed and investigated. 
In Section \ref{section-li}, a new unlinking operation, the circled region crossing change is introduced and the proof of Theorem \ref{thm-cru-one} is given.

\section{Region crossing change}
\label{section-rcc}

In this section, some properties of RCC on knot diagrams are discussed. 

\subsection{Changeable crossings}

\noindent A crossing $c$ of a knot diagram $D$ is said to be {\it changeable by RCC} if it is possible to change only $c$ by RCCs.
The following theorem was shown in \cite{rcc-uo}:

\vspace{2mm}

\begin{thm}[\cite{rcc-uo}]
Any crossing of a knot diagram is changeable by RCCs. 
\label{thm-rcc-cc-k}
\end{thm}

\vspace{2mm}

\noindent For links, the following theorem was shown in \cite{cheng-gao}: 

\vspace{2mm}

\begin{thm}[\cite{cheng-gao}]
For a link diagram, any self-crossing of a knot component is changeable by RCCs. 
Any crossing between different components is not changeable by RCCs. 
\label{thm-rcc-cc-l}
\end{thm}

\vspace{2mm}

\subsection{Prohibited and compulsory regions}

For a knot projection $P$, an {\it ineffective set} is a set $S$ of regions of $P$ such that RCCs at all the regions in $S$ cause no changes. 
For example, the set of all the shaded regions of an irreducible knot projection with a checkerboard coloring is an ineffective set. 
Recently, ineffective sets with two ``prohibited'' regions on a generalized Region Select game were discussed with some settings in \cite{k-col}, \cite{fukushima} and \cite{jong}. 
From their results, we have the following lemmas on RCCs on knot diagrams:

\vspace{2mm}

\begin{lem}[\cite{k-col}, \cite{fukushima}, \cite{jong}, \cite{rims}, etc.]
Let $R_1 , R_2$ be adjacent regions of a knot projection $P$. 
For each crossing $c$ of $P$, there exists a set $S$ of regions such that $S$ does not include $R_1$ nor $R_2$ and changes only $c$ by RCCs. 
\label{prop-p-c}
\end{lem}

\vspace{2mm}

\begin{lem}
Let $R_1 , R_2$ be adjacent regions of a knot projection $P$. 
There exists an ineffective set for $P$ which does not include $R_1$ nor $R_2$. 
\end{lem}

\vspace{2mm}

\noindent For ``prohibited'' and ``compulsory'' regions, we have the following: 

\vspace{2mm}

\begin{lem}
Let $R_1$, $R_2$ be adjacent regions of a knot projection $P$. 
There exists an ineffective set of $P$ which includes $R_1$ and does not include $R_2$. 
\label{prop-pc-in}
\end{lem}

\begin{proof}
Let $c_1 , c_2 , \dots ,c_n$ be the crossings on $\partial R_1$. 
Let $S_1 , S_2 , \dots ,S_n$ be sets of regions such that $S_i$ does not include $R_1$ nor $R_2$ and changes $c_i$ by RCCs (Lemma \ref{prop-p-c}). 
Take the symmetric difference $S= R_1 \oplus S_1 \oplus S_2 \oplus \dots \oplus S_n$, where $T_1 \oplus T_2 = (T_1 \cup T_2 ) \setminus (T_1 \cap T_2 )$ and $T_1 \oplus T_2 \oplus T_3 = (T_1 \oplus T_2) \oplus T_3$. 
Then, $S$, which is an ineffective set, includes $R_1$ and does not include $R_2$. 
\end{proof}

\vspace{2mm}

\begin{lem}
Let $R_1$, $R_2$ be adjacent regions of a knot projection $P$. 
For each crossing $c$ of $P$, there exists a set of regions $S$ which includes $R_1$, does not include $R_2$, and changes only $c$ by RCCs.
\label{prop-pc-c}
\end{lem}

\begin{proof}
Let $S^1$ be an ineffective set such that $R_1 \in S^1$ and $R_2 \not\in S^2$ (Lemma \ref{prop-pc-in}). 
Let $S^2$ be an ineffective set such that $R_1 \not\in S^2$ and $R_2 \in S^2$. 
Let $T$ be a set of regions of $P$ which changes only $c$ by RCCs. 
If $T$ does not include $R_1$, take the symmetric difference $T \oplus S^1$ with $S^1$, and/or if $T$ includes $R_2$, take the symmetric difference with $S^2$. 
\end{proof}

\section{Region Select on origami}
\label{section-ors}

In this section, Region Select on origami is introduced and investigated. 

\subsection{Crease pattern}

Let $S$ be a square sheet of paper. 
After a sequence of folding, $S$ has a ``crease pattern'', a union of creases on $S$. 
By considering each intersection of creases as a vertex and each crease segment as an edge, we can describe crease patterns in terms of graph theory. 
We note that there are some edges which have no vertices on $\partial S$. 
We define a face to be a connected portion of $S$ divided by the creases.

If it is possible to use all the creases to create a flat origami object, we say the crease pattern is a {\it flat foldable pattern}. 
The following necessary conditions are well-known in the study of origamis. (See, for example, \cite{hull}.)

\begin{prop}
Let $P$ be a flat foldable pattern. Then, the following holds. \\
1. Each vertex in $\mathrm{int}(S)$ has an even degree. \\
2. Around each vertex in $\mathrm{int}(S)$, the sum of degrees of every other angles is $180^{\circ}$ .
\end{prop}

\vspace{2mm}
\noindent We note that the above two conditions are not sufficient conditions.

\subsection{Region Select on crease pattern}

Let $P$ be a flat foldable crease pattern on a square sheet of paper $S$. 
Give a lamp which has the status of ON or OFF to each vertex. 
Then we can play the Region Select on $P$, and also we can discuss the changeability by RCCs for crossings of creases in the same manner to knot projections. 
For example, on the crease pattern in Figure \ref{fig-deg6} with crossings labeled $v_1, v_2, \dots , v_7$ and regions labeled $R_1, R_2, \dots , R_{12}$, we can see that the crossing $v_2$ is changeable by RCCs at $\{ R_9, R_{12} \}$, $\{ R_2, R_5, R_6, R_7, R_9, R_{11} \}$, etc. We can see that the crossing $v_1$ is unchangeable by RCCs by solving the following linear system (see \cite{ahara-suzuki}, \cite{cheng-gao} for the linear systems): 

\begin{align*}
\left\{
\begin{array}{l}
v_1 : x_1 + x_5 + x_7 + x_8 \equiv 1 \pmod 2 \\
v_2 : x_4 + x_5 + x_6 + x_9 + x_{10} \equiv 0 \pmod 2 \\
v_3 : x_3 + x_6 + x_{10} + x_{11} \equiv 0 \pmod 2 \\
v_4 : x_1 + x_2 + x_4 + x_5 \equiv 0 \pmod 2 \\
v_5 : x_2 + x_3 + x_4 + x_6 \equiv 0 \pmod 2 \\
v_6 : x_7 + x_8 + x_9 + x_{12} \equiv 0 \pmod 2 \\
v_7 : x_9 + x_{10} + x_{11} + x_{12} \equiv 0 \pmod 2 
\end{array}
\right.
\end{align*}

\begin{figure}[h]
\centering
\includegraphics[width=25mm]{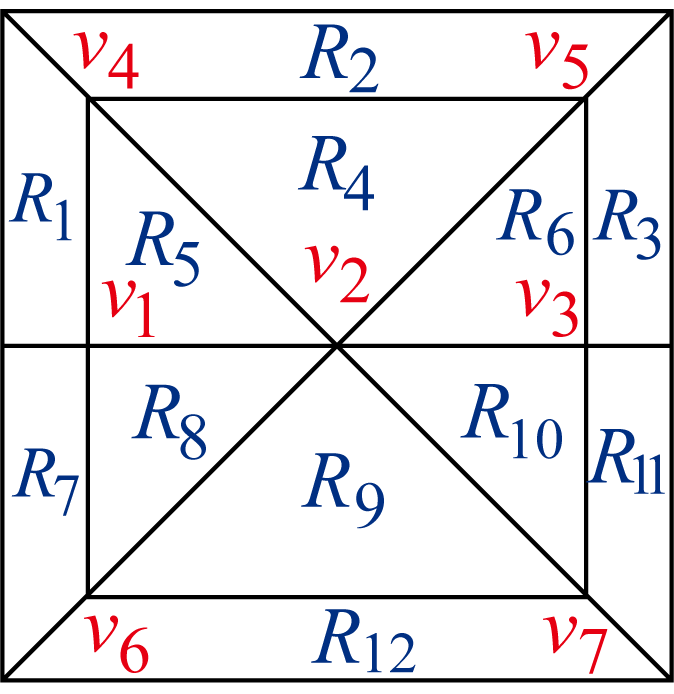}
\caption{The crossing $v_1$ is unchangeable and $v_2$ is changeable.}
\label{fig-deg6}
\end{figure}

\noindent We can confirm that the linear system has no solutions by summing the 1, 3, 4, 5, 6, 7th equations. 
In the next two subsections, we will focus on the ``4-regular'' cases with two types of the situations on the boundary of $S$.

\subsection{Region Select on 4-regular crease pattern, type I}

Let $P$ be a flat foldable pattern on a square sheet of paper $S$ such that every vertex has degree four and there are no intersections of creases on $\partial S$. 
Here, since $P$ is flat foldable, each angle around a vertex is less than $180^{\circ}$ and then each face is a convex shape.

A {\it tanglize} is the following transformation of a 4-regular crease pattern into a ``tangle''; 
Take a start point $p$ on an edge of $P$. 
Travel the edge from $p$ in both directions. 
When meeting a vertex, proceed to the edge on the opposite side. 
Keep traveling until meeting $\partial S$ or the path is closed on an edge (not on a vertex). 
By the above procedure, we obtain a (closed or open) curve uniquely. 
Repeat the procedure until every crease is done. 
Then we obtain a tangle $T$ with components $K_1, K_2, \dots , K_l, L_1, L_2, \dots ,L_m$, where a component $K_i$ implies a closed curve and $L_i$ implies an open curve (see Figure \ref{fig-tanglize}). 
We have the following proposition. 

\begin{figure}[h]
\centering
\includegraphics[width=50mm]{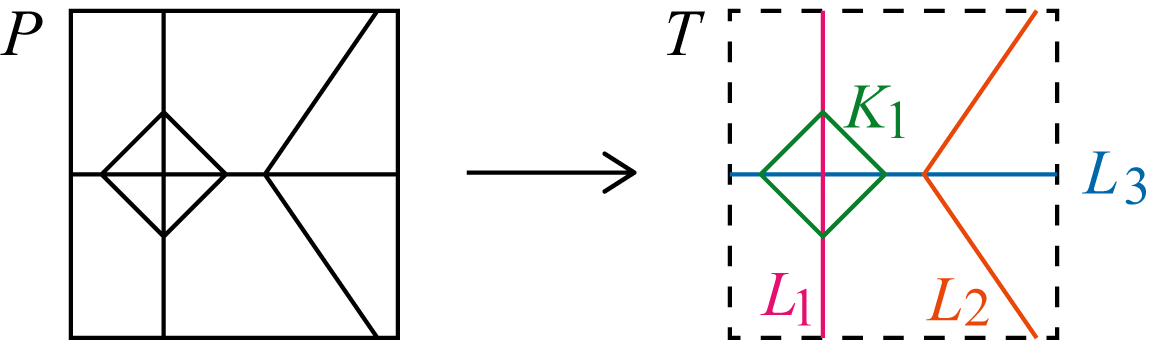}
\caption{Tanglize.}
\label{fig-tanglize}
\end{figure}

\begin{prop}
Each self-crossing $c$ of $K_i$ is changeable by RCCs\footnote{At the moment, the authors have no examples of a flat foldable crease pattern such that $K_i$ has a self-crossing.}. 
\label{lem-ki}
\end{prop}

\vspace{2mm}

\noindent Before the proof, we discuss the reducible crossings of a tangle. 
A {\it reducible crossing} of a tangle $T$ is the crossing such that we can put a simple closed curve inside the tangle so that the simple closed curve intersects only $c$ transversely. 
We call such a simple closed curve a {\it reducing simple closed curve for $c$}. 
Let $c_1, c_2, \dots , c_l$ be all the reducible crossings in a tangle $T$. 
We give a partial order $c_{\alpha} \preceq c_{\beta}$ in the following rule. 
Let $c_{\alpha}, c_{\beta}$ be reducible crossings of the same component or intersecting two components. 
If one can draw reducing simple closed curves $A$ and $B$ for $c_{\alpha}$ and $c_{\beta}$, respectively, so that $A$ is outside of $B$, we define $c_{\alpha} \preceq c_{\beta}$. 
If two components belong to split components, their partial order is not defined.
For example, in Figure \ref{fig-red-c}, $c_2 \preceq c_1$, $c_2 \preceq c_3$ and $c_4 \preceq c_5$. 
We remark that the partial order between $c_2$ and $c_7$ is undefined since they are on split components. 
Using a method in \cite{rcc-uo}, we prove Proposition \ref{lem-ki}. \\

\begin{figure}[h]
\centering
\includegraphics[width=23mm]{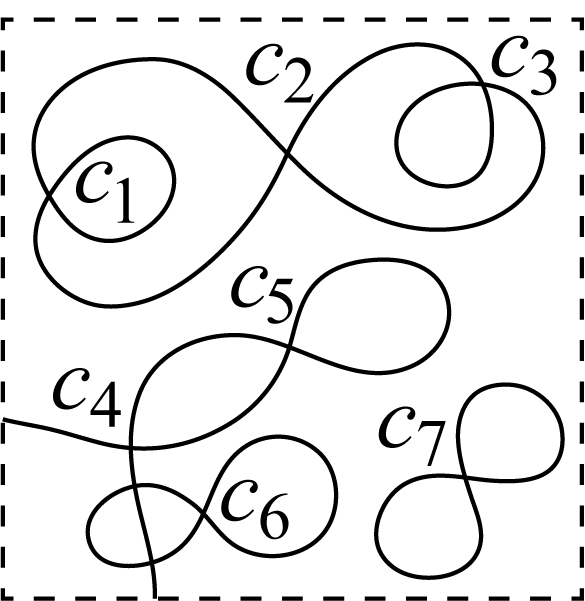}
\caption{Reducible crossings.}
\label{fig-red-c}
\end{figure}

\vspace{2mm}

\noindent {\it Proof of Proposition \ref{lem-ki}.}
Looking at only $K_i$, take the following procedure. 
Splice $K_i$ at $c$ to obtain a two-component link. 
Give a checkerboard coloring to one of the two components. 
Take a set $S$ of regions of $T$ which correspond to the shaded regions. 
Then, there are even number of regions of $S$ around irreducible self-crossings of $K_i$ except $c$, irreducible self-crossings of other components, and any crossings between different components. 
Namely, any irreducible crossing except for $c$ is unchanged by the RCCs at $S$. 
On the other hand, for reducible crossings, some of them may have odd number of regions which belong to $S$ (see $c_1, c_2$ and $c_3$ in Figure \ref{fig-ki-alg}). \\
For reducible crossings, apply the following modification to obtain a set of regions which changes only $c$ by RCCs. 
Let $c_1, c_2, \dots , c_l$ be all the reducible crossings of $T$. 
We note that the crossing $c$ is a member of them when $c$ is a reducible crossing. 
Reorder the reducible crossings if necessary so that they do not contradict with the partial order, and apply the following procedure from $c_1$ to $c_l$. 
If $c_i$ has an undesired status, splice at $c_i$, choose a knot component, and apply a checkerboard coloring such that the outer region is unshaded (see Figure \ref{fig-ki-mod}). 
Take the set $S_i$ of regions of $T$ which correspond to the shaded regions, and take the symmetric difference of $S_i$ and the previous set of regions. 
Thus, we obtain the set of regions which changes only $c$ by RCCs. 
\qed

\begin{figure}[h]
\centering
\includegraphics[width=92mm]{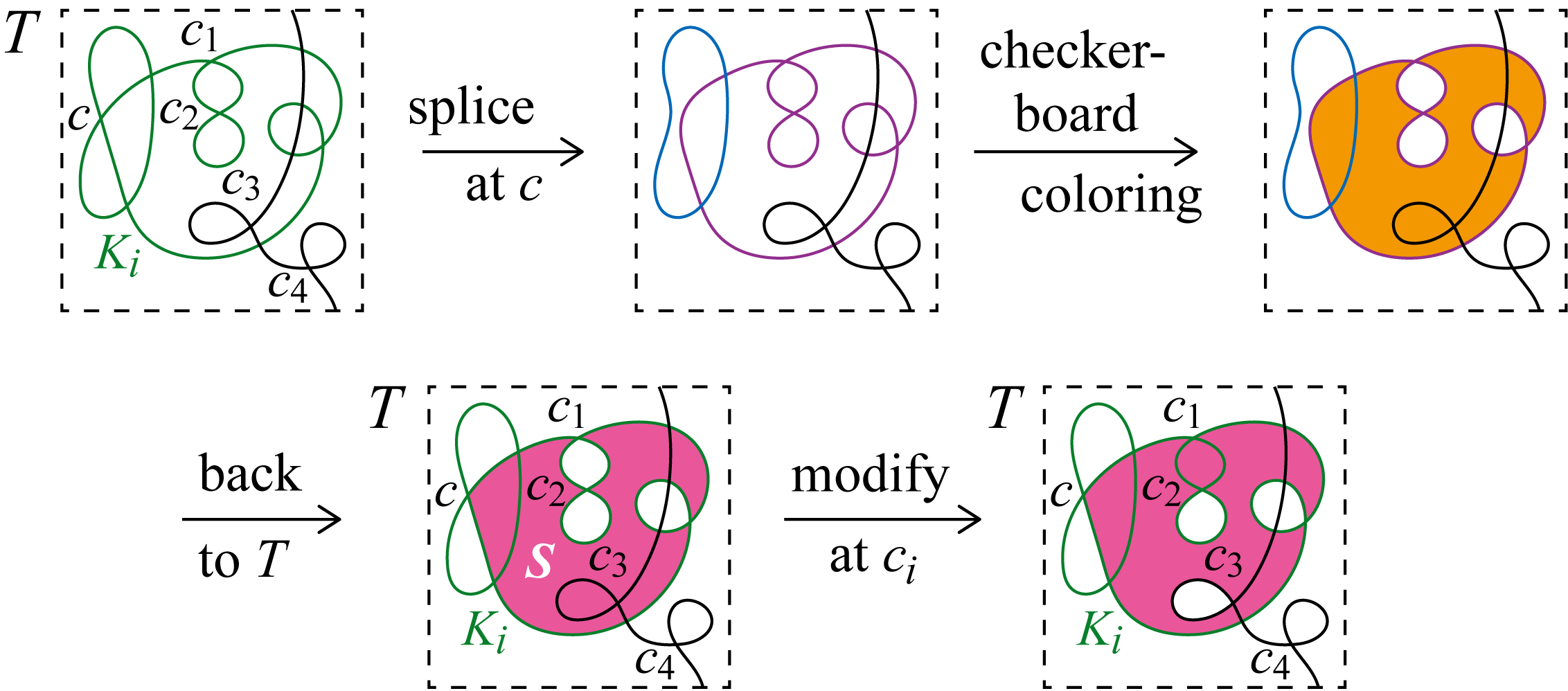}
\caption{The procedure for the proof of Proposition \ref{lem-ki}.}
\label{fig-ki-alg}
\end{figure}

\begin{figure}[h]
\centering
\includegraphics[width=18mm]{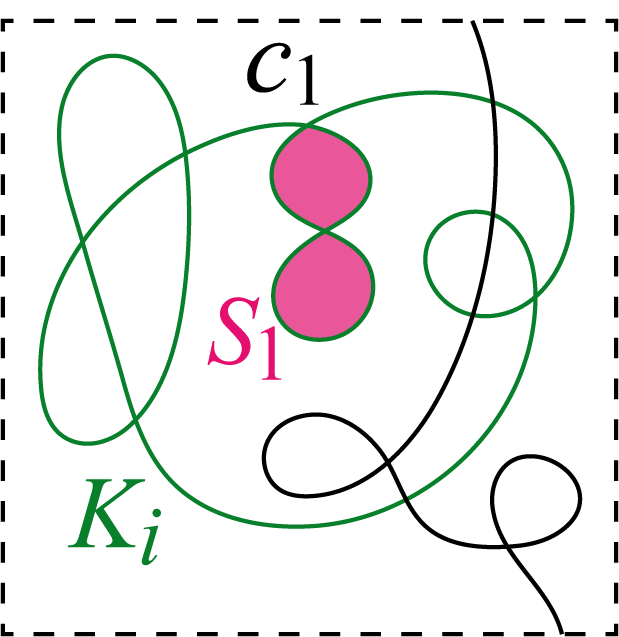}
\caption{$c_1$ and $S_1$.}
\label{fig-ki-mod}
\end{figure}

\noindent We also have the following propositions. 

\vspace{2mm}

\begin{prop}
Each self-crossing of $L_i$ is changeable by RCCs. 
\label{lem-li}
\end{prop}

\begin{proof}
Since $L_i$ is an open path, we obtain a knot projection by connecting the endpoints outside $S$ (see Figure \ref{fig-1-tangle-cc}).
In the same way to the proof of Proposition \ref{lem-ki}, splice at $c$, give a checkerboard coloring to one of the two components, and apply modifications at some reducible crossings. 
\end{proof}

\begin{figure}[h]
\centering
\includegraphics[width=70mm]{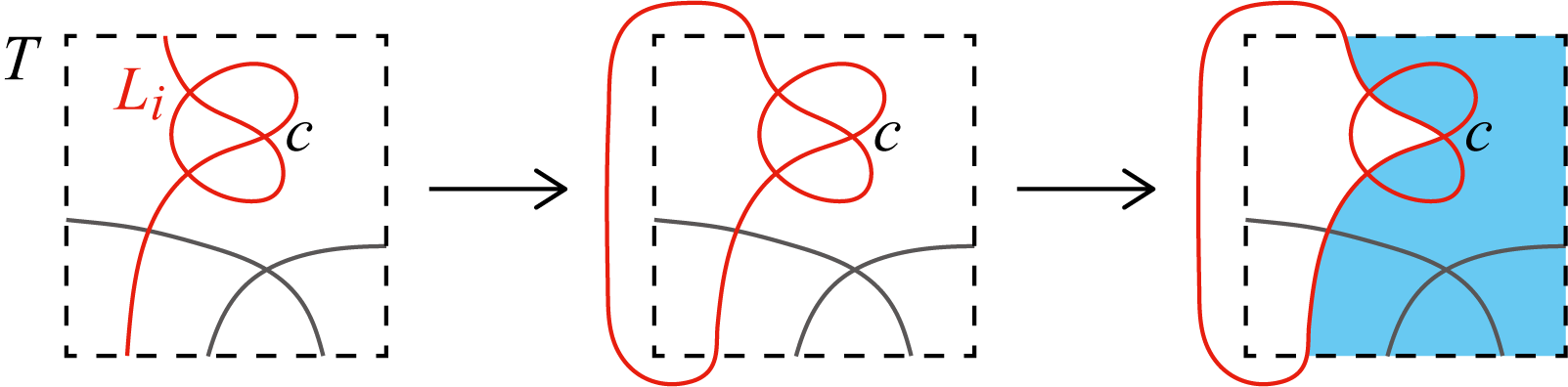}
\caption{Any self-crossing of $L_i$ is changeable by RCCs.}
\label{fig-1-tangle-cc}
\end{figure}

\vspace{2mm}

\begin{prop}
Each crossing between $L_i$ and $L_j$ ($i \neq j$) is changeable by RCCs.
\label{lem-lilj}
\end{prop}

\begin{proof}
Cut $L_i$ and $L_j$ at the crossing $c$. 
By connecting ones of the cut paths of $L_i$ and $L_j$, we obtain another open path and then obtain a knot projection by connecting the endpoints outside $S$. 
Take an ineffective set of the knot projection which includes one of the regions touching $c$ and does not include the other region touching $c$ (Lemma \ref{prop-pc-in}). 
Then, the set changes only $c$ by RCCs, after some modifications at reducible crossings if necessary. 
See Figure \ref{fig-pi-pj-cc}. 
\end{proof}

\begin{figure}[h]
\centering
\includegraphics[width=27mm]{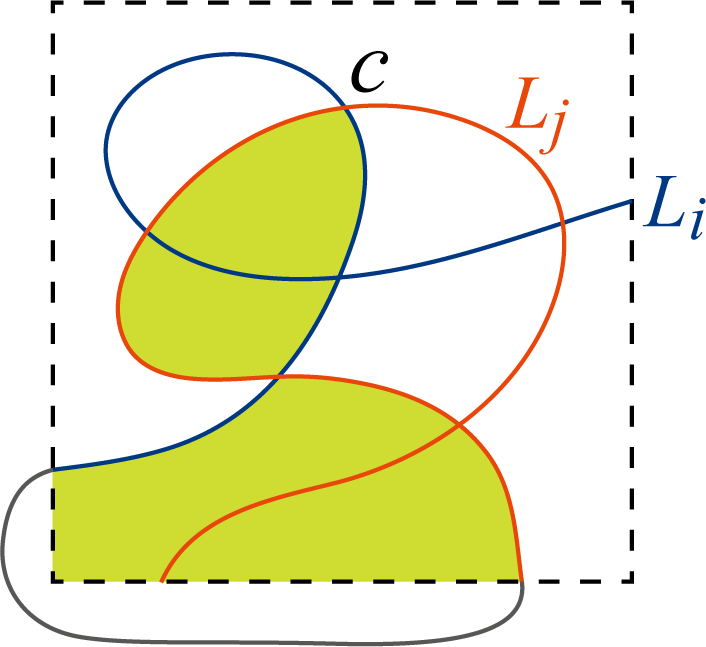}
\caption{Cut $L_i$ and $L_j$ and create another open path.}
\label{fig-pi-pj-cc}
\end{figure}

\noindent We note that when the tangle includes no closed curves, Propositions \ref{lem-li} and \ref{lem-lilj} follow results in \cite{HSS} on diagrams of a connected spatial graphs by considering the tanglized crease pattern as a diagram of a connected trivalent spatial graph. 
We have the following proposition.

\vspace{2mm}

\begin{prop}
Any single crossing between $K_i$ and another component is unchangeable by RCCs.
\label{lem-ki-other}
\end{prop}

\vspace{2mm}

\noindent To prove Proposition \ref{lem-ki-other}, we define the {\it lamp-linking number} of $K_i$, denoted by $l(K_i)$. 
At first, we give a value to each crossing $c$ with lamp as follows: 
If $c$ is a self-crossing of $K_i$, give $f(c)=0$. 
If $c$ is a crossing of the other components, give $f(c)=0$. 
If $K_i$ is a crossing between $K_i$ and another component and the lamp is on, give $f(c)=1$. 
If $K_i$ is a crossing between $K_i$ and another component and the lamp is off, give $f(c)=-1$. 
We define the lamp-linking number $l(K_i)$ of $K_i$ to be $1/2 \sum_{c} f(c)$. 
For example, the crease pattern with lamps in Figure \ref{fig-lamp-linking} has lamp-linking number $(-1+1-1+1)/2 =0$. 
Using Cheng and Gao's method in \cite{cheng-gao}, we show Proposition \ref{lem-ki-other}.

\begin{figure}[h]
\centering
\includegraphics[width=25mm]{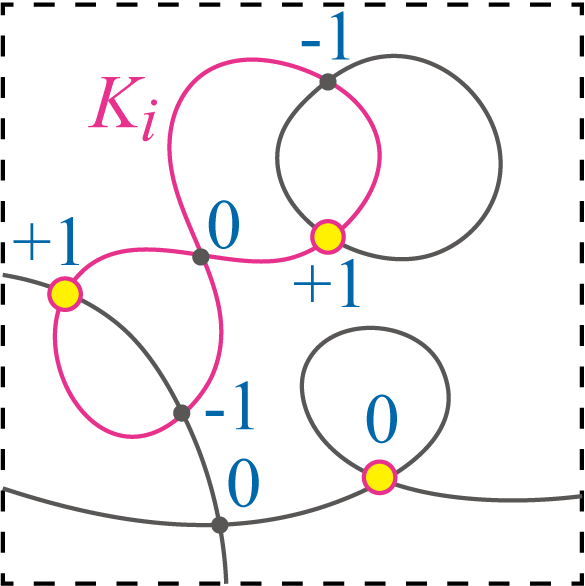}
\caption{The lamp-linking number is zero.}
\label{fig-lamp-linking}
\end{figure}

\vspace{4mm}
\noindent {\it Proof of Proposition \ref{lem-ki-other}.}
The number of the crossing of $K_i$ and others is an even number on any regions of $P$. 
Therefore, the parity of $l(K_i)$ is unchanged by any RCC on $P$. 
Hence, we cannot make a single crossing change between $K_i$ and the others by RCCs, which changes the parity of $l(K_i)$. 
\qed

\subsection{Region Select on 4-regular crease pattern, type II}

In this subsection, we consider 4-regular flat foldable patterns $P$ on a square $S$ which may have vertices on the boundary of $S$. 
Here, since $P$ is flat foldable, each angle around a vertex is less than $180^{\circ}$ and then we can assume that each vertex on $\partial S$ has at most three edges inside $S$. 
We also assume that there are no vertex on $\partial S$ which has only one edge in $\mathrm{int}(S)$.

A {\it generalized tanglize} of a crease pattern of this type is to be defined as the following procedure to obtain a generalized tangle: 
Take a start point on an edge and travel in both directions. 
When meeting a 4-valent vertex, proceed to the edge in the opposite side. 
When meeting a vertex on $\partial S$ which has two edges, proceed to another edge. 
When meeting a vertex on $\partial S$ which has three edges and one comes along a boundary-side edge, proceed to another edge on boundary-side. 
When meeting a vertex on $\partial S$ which has three edges and one comes along the center edge, stop at the vertex. 
Keep traveling until one has come to two end points on $\partial S$ or the path is closed on an edge. 
Thus, we obtain a set of open or closed paths in $S$, where some of the paths may touch $\partial S$. 
We call this a {\it contact-tangle} $T$ with components $K_1, K_2, \dots , K_l, L_1, L_2, \dots , L_m$, where $K_i$ is a closed path and $L_i$ is an open path. 
We may call $K_i$ a knot component as well. 
See Figure \ref{fig-contact-tangle} for examples. 

\begin{figure}[h]
\centering
\includegraphics[width=110mm]{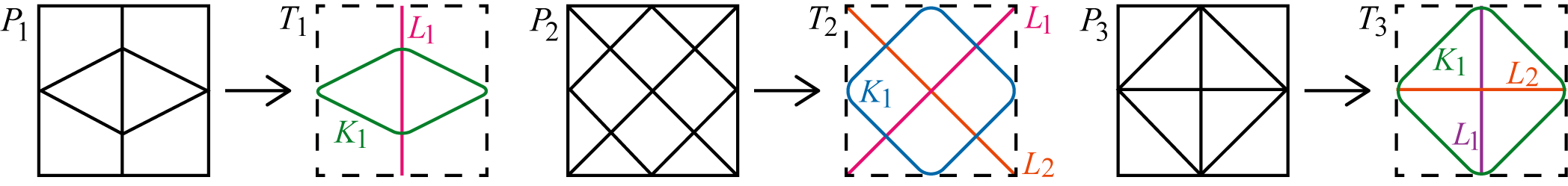}
\caption{Generalized tanglize.}
\label{fig-contact-tangle}
\end{figure}

\vspace{2mm}

For a contact-tangle with lamps, we also define the {\it lamp-linking number}. 
Let $K_i$ be a knot component of a contact-tangle $T$. 
Let $c$ be a crossing or a contacting point on $\partial S$. 
Give a value to each $c$ with lamp as follows: 
If $c$ is a self-crossing or a crossing of the other components than $K_i$, $f(c)=0$. 
If $c$ is a crossing between $K_i$ and another component and the lamp is ON (resp. OFF), $f(c)=1$ (resp. $f(c)=-1$). 
If $c$ is a contact point of $K_i$ with two edges and the lamp is ON (resp. OFF), $f(c)=1$ (resp. $f(c)=-1$). 
If $c$ is a contacting point of $K_i$ with three edges, follow the above rule. 
If $c$ is a contacting point which is not involved with $K_i$, $f(c)=0$. 
We define the lamp-linking number, $l(K_i)$, of $K_i$ to be the value $1/2 \sum _c f(c)$. 

A knot component $K_i$ in a contact-tangle $T$ is said to be an {\it even component} if the total number of crossings between $K_i$ and other components and contact points between $K_i$ and $\partial S$ is even on the boundary for any region of $T$. 
In Figure \ref{fig-contact-tangle}, for example, $T_1$ and $T_2$ have no even components whereas $T_3$ has an even component $K_1$. 
We have the following. 

\vspace{2mm}

\begin{prop}
When $K_i$ is an even component, neither a crossing between $K_i$ and another component nor a contact point of $K_i$ on $\partial S$ is changeable by RCC. 
\end{prop}

\begin{proof}
In this case each RCC changes the lamp-linking number of $K_i$ by $2n$ for some $n \in \mathbb{Z}$. 
Hence, any switching of a lamp between $K_i$ and another components or $\partial S$, which changes the lamp-linking number by $\pm 1$, cannot be realized by RCCs.
\end{proof}

\begin{figure}[h]
\centering
\includegraphics[width=40mm]{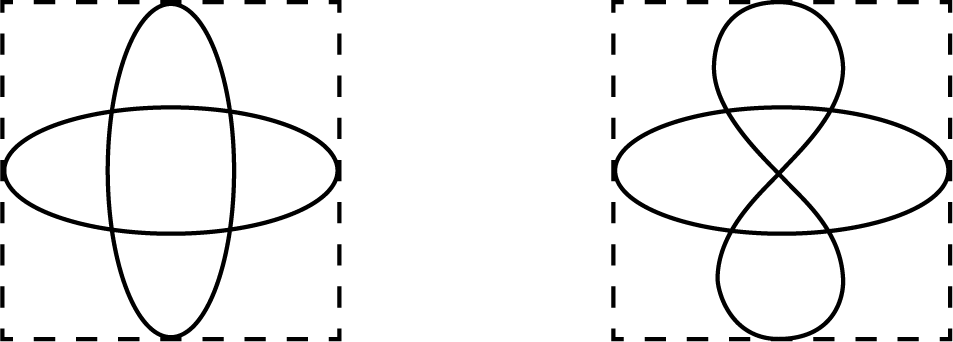}
\caption{The contact-tangles both have no even components. On the contact-tangle on the left-hand side, any crossing or touching point is changeable by RCCs, whereas the other one has unchangeable crossings.}
\label{fig-on-even}
\end{figure}

\vspace{2mm}

\noindent We have some examples of contact-tangles which has non-even components but has unchangeable crossings. 
In Figure \ref{fig-contact-tangle}, we cannot change some lamps of $T_2$, whereas we can change any lamp of $T_1$. 
See also Figure \ref{fig-on-even}. 
For a knot component touching at just one point, we have the following. 

\vspace{2mm}

\begin{prop}
Let $K_i$ be a knot component of $T$ touching $\partial S$ at just one point $c_1$. 
If $c_1$ has exactly two edges inside $S$, then $c_1$ is changeable by RCCs. 
\label{lem-touch-one}
\end{prop}

\begin{proof}
When we ignore other components, there are just two regions of $K_i$ which touch $c_1$. 
We call the region which is not bounded by $\partial S$ the region $R_1$, and the other one $R_2$. 
Take the ineffective set of $K_i$ which includes $R_1$ and does not include $R_2$ (Lemma \ref{prop-pc-in}). 
By the RCCs on the regions (and some modifications at reducible crossings of $K_i$ or $T$), only $c_1$ is changed. 
We remark that the region $R_2$ is divided in $T$ because  any flat foldable crease pattern does not have non-convex face. 
\end{proof}

\vspace{2mm}

\noindent For self-crossings, we have the following.

\vspace{2mm}

\begin{prop}
Let $K_i$ be a knot component of $T$ touching $\partial S$ at just one point $c_1$. 
If $c_1$ has exactly two edges inside $S$, then any self-crossing $c$ of $K_i$ is changeable by RCCs. 
\end{prop}

\begin{proof}
Let $U$ be a set of regions of $K_i$ which changes only a self-crossing $c$ of $K_i$ by RCCs. 
If $U$ changes $c_1$, not only $c$, by the RCCs on $T$, then apply the RCCs at the regions of Proposition \ref{lem-touch-one}.
\end{proof}

\section{Circled region unlinking number}
\label{section-li}

\noindent In this section, using the concept of the origami Region Select, a new unlinking number associated with RCC is defined. 
Since a RCC is an unknotting operation on a knot diagram, the region unknotting number $u_R(D)$ is well-defined to be the minimal number of RCCs to transform $D$ into a diagram of a trivial knot. 
By Theorem \ref{thm-rcc-cc-l} of Cheng and Gao, however, some types of link diagrams cannot be unlinked by any RCCs. 
Moreover, Cheng gave a necessary and sufficient condition for links to be unlinked by RCCs\footnote{See \cite{FNS} for spatial-graphs.}: 

\vspace{2mm}

\begin{thm}[\cite{cheng}]
A diagram of a link $L$ can be transformed into a diagram of the trivial link by RCCs if and only if $L$ is proper \footnote{A {\it proper link} is a link such that the sum of the linking numbers between a component and other components takes an even number for any component. For example, a Hopf link is an improper link and the Borromean ring is a proper link.}. 
\end{thm}

\vspace{2mm}

\noindent Using the idea of the origami Region Select, we define a new operation which can unlink any link, even for improper links. 
Let $D$ be a link diagram on $S^2$. 
Give a simple closed curve $C$ on $S^2$ so that $C$ intersects arcs of $D$ transversely avoiding the crossing points of $D$. 
Then, $S^2$ is divided into two discs, say $S_N$ and $S_S$, and $D$ is also divided into two tangles, $D_N$ on $S_N$ and $D_S$ on $S_S$. 
We will define a RCC on $D_N$ or $D_S$ in a same manner to the origami Region Select. 
A {\it region} of $D$ with $C$ is a connected portion of $S^2$ bounded by arcs of $D$ or $C$, and a RCC on a region is a set of crossing changes at the crossings on the boundary of the region.  
The {\it circled region unlinking number}, $u_{\textcircled{R}}(D, C)$, of $D$ with $C$ is the minimal total number of RCCs on $D_N$ and $D_S$ required to transform $D$ into a diagram of a trivial link. 
The {\it circled region unlinking number}, $u_{\textcircled{R}}(D)$, of $D$ is the minimal value of $u_{\textcircled{R}}(D, C)$ for all $C$ (See Figure \ref{fig-ur-two}). 
We have the following proposition.

\begin{figure}[h]
\centering
\includegraphics[width=25mm]{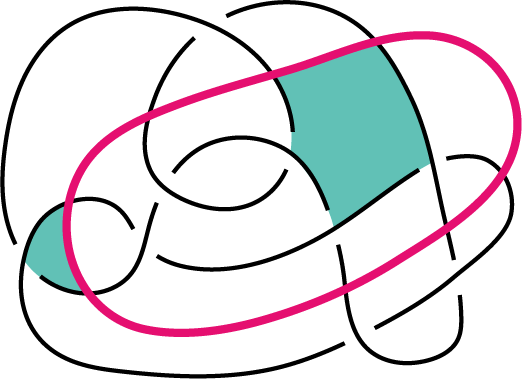}
\caption{The diagram $D$ of the link $9^2_{22}$, or $L9a23$, has $u_{\textcircled{R}}(D, C)=2$ with the circle $C$. We note that $u(9^2_{22})=4$ and $9^2_{22}$ is an improper link.}
\label{fig-ur-two}
\end{figure}

\vspace{2mm}
\begin{prop}
The circled region unlinking number $u_{\textcircled{R}}(D)$ takes a finite number for any link diagram $D$. 
\label{prop-finite}
\end{prop}

\vspace{2mm}

\begin{proof}
When $D$ is a knot diagram, $D$ is unknotted by some RCCs with any $C$. 
Let $D$ be a diagram of an $n$-component link ($n \geq 2$).  
On $D$, take a base point on an edge for each knot component. 
Connect the $n$ base points by a simple path $P$. 
Take the boundary $C$ of a neighborhood of the path with suitable thickness so that $C$ intersects arcs of $D$ transversely. 
Then, $C$ bounds tangles on both sides which have no closed curves. 
By Proposition \ref{lem-li}, we can make any crossing change in the tangles. 
\end{proof}

\begin{figure}[h]
\centering
\includegraphics[width=60mm]{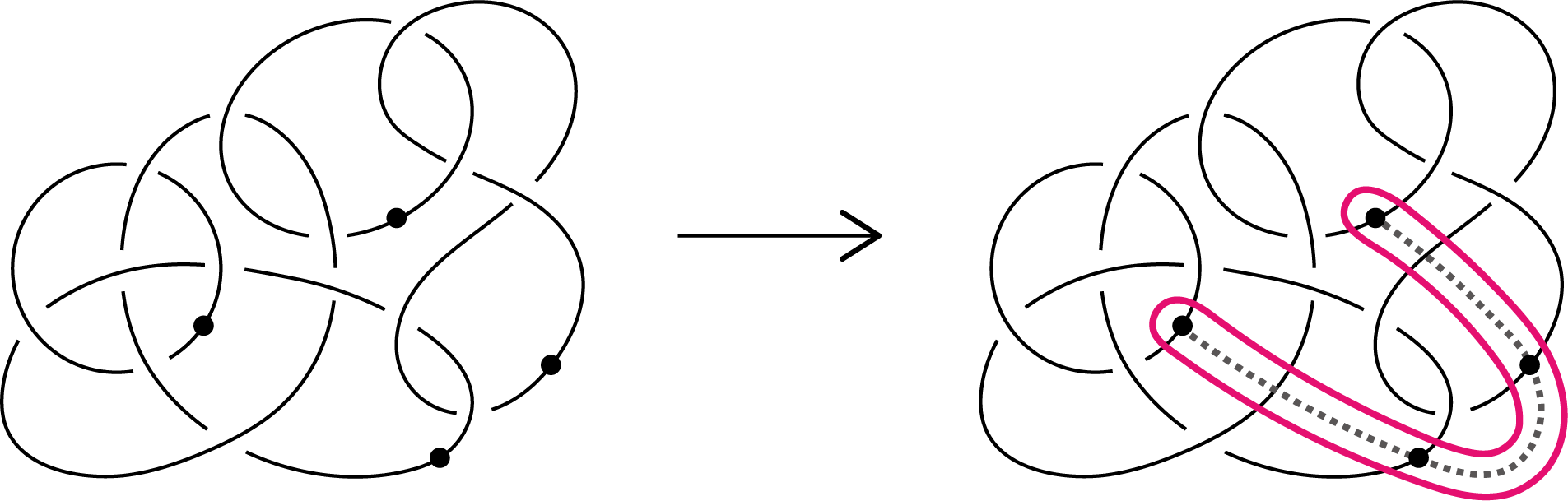}
\caption{Take a base point on an edge for each component and connect the points by a simple path $P$, as shown by the broken curve. Take the boundary $C$ of a neighborhood of $P$, as shown by the pink curve. }
\label{fig-comp-path}
\end{figure}

\vspace{2mm}

\noindent Let $u(D)$ denote the classical unlinking number of $D$, namely, the minimal number of crossing changes which are required to transform $D$ into a diagram of the trivial link. 

\vspace{2mm}

\begin{prop}
The inequality $u_{\textcircled{R}}(D) \leq u(D)$ holds for any link diagram $D$. 
\label{prop-uru}
\end{prop}

\begin{proof}
Let $c_1 , c_2 , \dots , c_k$ be crossing points of $D$ such that the crossing changes at the crossings transform $D$ into a diagram of a trivial link. 
When $k \geq 2$, in the same way to the proof of Proposition \ref{prop-finite}, connect the $k$ crossing points by a simple path $P$, and take the boundary $C$ of a neighborhood of $P$ with appropriate thickness. 
When $k=1$, just take a small circle around $c_1$. 
Then, in the neighborhood-side of $C$, we can make each crossing change by a single RCC. 
\end{proof}

\begin{figure}[h]
\centering
\includegraphics[width=60mm]{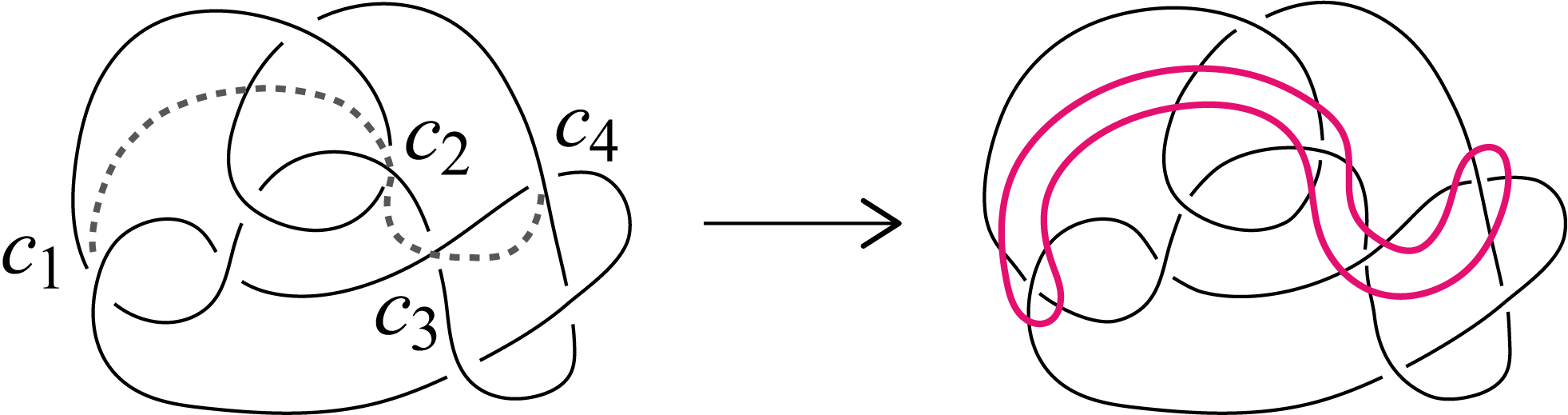}
\caption{Take the boundary of the neighborhood of the path.}
\label{fig-cc-path}
\end{figure}

\vspace{2mm}

\noindent In particular, we have the following.

\vspace{2mm}

\begin{cor}
If $u(D)=1$, then $u_{\textcircled{R}}(D)=1$.
\end{cor}

\vspace{2mm}

\noindent We note that we have lots of examples satisfying $u_{\textcircled{R}}(D)=1$ but $u(D) \neq 1$. 
See Figure \ref{fig-ur-one}. 

\begin{figure}[h]
\centering
\includegraphics[width=80mm]{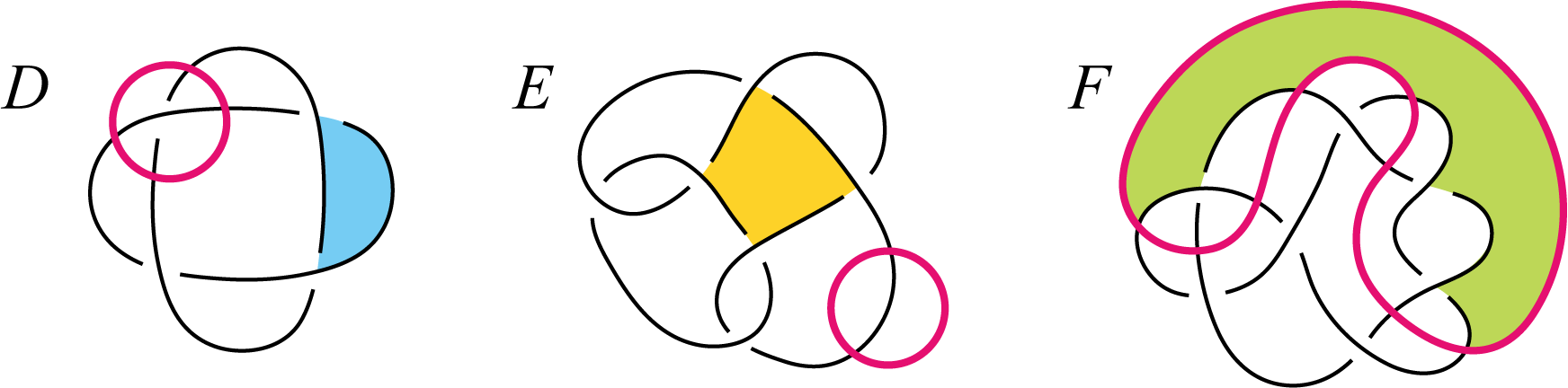}
\caption{The diagrams $D, E$ and $F$ are diagrams of the links $4^2_1$, $6^2_3$ and $7^2_1$, respectively. 
We have $u_{\textcircled{R}}(D)= u_{\textcircled{R}}(E)= u_{\textcircled{R}}(F)=1$ and $u(D)=u(E)=u(F)=2$. 
We note that $4^2_1$ and $6^2_3$ are proper links and $7^2_1$ is an improper link (see, for example, \cite{kohn}).}
\label{fig-ur-one}
\end{figure}

\noindent The inequality $u(D) \leq \frac{c(D)}{2}$ is well-known since the mirror image of a diagram of a trivial link also represents the trivial link. 
We have the following corollary\footnote{For knot diagrams $D$, recently it is shown in \cite{fah} that $u_R(D) \leq \frac{1}{2}(c(D)+1)$ for the classical region unknotting number $u_R(D)$.} from Proposition \ref{prop-uru}. 

\vspace{2mm}

\begin{cor}
The inequality $u_{\textcircled{R}}(D) \leq \dfrac{c(D)}{2}$ holds for each link diagram $D$. 
\end{cor}

\vspace{2mm}

\noindent Using the method of the ``spur'' used in \cite{FNS} and \cite{rinno}, we prove Theorem \ref{thm-cru-one}.

\vspace{4mm}

\noindent {\it Proof of Theorem \ref{thm-cru-one}.} 
Let $D$ be a diagram of a link $L$ such that $u(D)=n$. 
Take an arbitrary region $R$ of $D$, and draw a circle $C$ in $R$. 
Let $c_1, c_2, \dots c_n$ be $n$ crossings which transform $D$ into a diagram of a trivial link by the crossing changes. 
Draw $n$ paths from $c_i$ to inside $C$ so that they have no crossings each other, and have no intersections with crossing points of $D$. 
Then, move the crossings into the inside of $C$ along the paths passing over arcs of $D$, as shown in Figure \ref{fig-cru-1}. 
Then, we can change the $n$ crossings all together by one RCC in $C$. 
\qed

\begin{figure}[h]
\centering
\includegraphics[width=120mm]{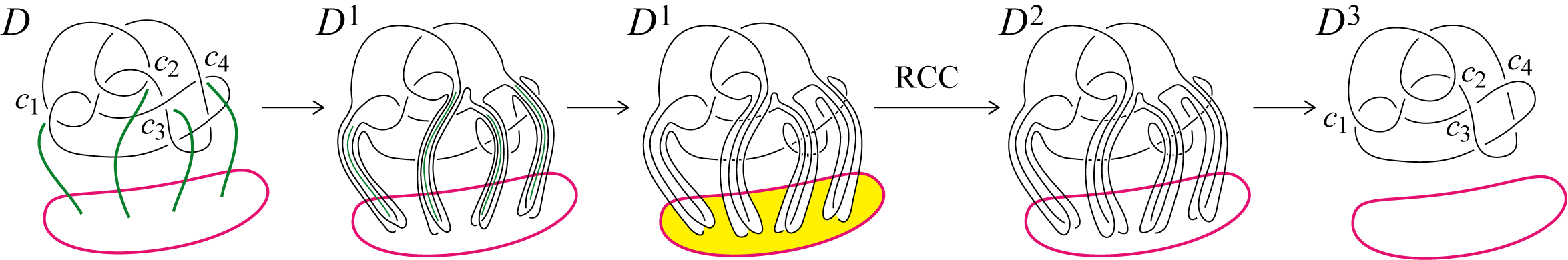}
\caption{For a link diagram $D$, move the crossings along the paths into $C$ to obtain a diagram $D^1$ such that the crossings are changed by a single RCC in $C$. 
After the RCC, we have a diagram $D^2$, which represents the same link to the diagram $D^3$ which is the diagram obtained from $D$ by crossing changes at these crossings. }
\label{fig-cru-1}
\end{figure}

\section*{Acknowledgment}
The second author's work was partially supported by JSPS KAKENHI Grant Number JP21K03263.

\end{document}